\documentclass[12pt,reqno]{amsart}
\usepackage{amsfonts, amsbsy, amsmath, amssymb, mathtools, amsthm,ragged2e}
\usepackage{amscd}
\usepackage{color,enumerate}
\usepackage{breqn}
\usepackage{bm}
\usepackage{pgfplots}
\pgfplotsset{compat=1.15}
\usepackage{tikz}
\usepackage{tikz-cd}
\usepackage{subcaption}
\usetikzlibrary{arrows}
\usepackage{verbatim}
\usepackage{graphicx}
\usetikzlibrary{matrix,calc}
\usepackage{caption}
\usepackage{hyperref}
\usepackage{float}
\usepackage[mathscr]{euscript}
\usepackage{mathrsfs}
\usepackage{cite}
\usepackage[margin=2.5cm]{geometry}
\usepackage{epsfig}

\newtheorem{thm}{Theorem}[section]
\newtheorem{lemma}[thm]{Lemma}

\newtheorem{corollary}[thm]{Corollary}

\newtheorem{conj}[thm]{Conjecture}
\newtheorem{thm-con}[thm]{Theorem-Conjecture}
\numberwithin{equation}{section}

\theoremstyle{definition}
\newtheorem{defn}[thm]{Definition}
\newtheorem{rmk}[thm]{Remark}
\newtheorem{notation}[thm]{Notation}

\newtheorem{exmp}[thm]{Example}

\newtheorem{set}[thm]{Set-up}

\DeclareMathOperator{\height}{ht}
\DeclareMathOperator{\Ass}{Ass}
\DeclareMathOperator{\ann}{ann}
\DeclareMathOperator{\lcm}{lcm}
\DeclareMathOperator{\supp}{supp}
\newcommand{\h}{\mathcal{H}}
\newcommand{\e}{\mathcal{E}}
\newcommand{\al}{\alpha}
\newcommand{\q}{\mathfrak{q}}
\newcommand{\p}{\mathfrak{p}}

\begin{document}
\title[Waldschmidt constant of monomial ideals and Simis ideals]{Waldschmidt constant of monomial ideals and Simis ideals}

\author{Bijender}
\email{bijender@lnmiit.ac.in}

\address{Department of Mathematics, The LNM Institute of Information Technology, Jaipur, India - 302031.}

\author[Ajay Kumar]{Ajay Kumar}
\email{ajay.kumar@iitjammu.ac.in}

\address{Department of Mathematics, Indian Institute of Technology Jammu, J\&K, India - 181221.}

\date{\today}



\maketitle	

\begin{abstract}
    In 2017, Cooper et al.\ proposed a conjecture providing a lower bound for the Waldschmidt constant of monomial ideals. We confirm this conjecture for some classes of monomial ideals. Recently, Méndez, Pinto, and Villarreal formulated a conjecture stating that if $I$ is a monomial ideal without embedded associated primes, whose irreducible decomposition is minimal and which is a Simis ideal, then there exist a Simis squarefree monomial ideal $J$ and a standard linear weighting $w$ such that $I = J_{w}.$ In this work, we verify this conjecture for some classes of monomial ideals.
\end{abstract}

\section{Introduction}
The study of comparison between symbolic powers and ordinary powers of a homogeneous ideal plays a crucial role in resolving major questions in commutative algebra and algebraic geometry. Unlike ordinary powers, which lack a direct geometric meaning, symbolic powers possess rich and attractive geometric properties. Yet, determining symbolic powers algebraically is typically quite challenging. These considerations motivate the fundamental problem of comparing symbolic and ordinary powers. Let $I \subset R = K[\mathbb{P}^{n-1}] = K[x_1, \dots, x_{n}], n \ge 2$, be a homogeneous ideal. For a positive integer $s \in \mathbb{N}$, the $s$-th symbolic power of $I$, denoted by $I^{(s)}$, is defined as 
$$I^{(s)} = \bigcap_{\p \in \Ass(I)} \left( I^{s} R_{\p} \cap R \right),$$
where $\Ass(I)$ is the set of associated primes of $I$ and $R_{\p}$ denotes the localization of $R$ at the prime ideal $\p.$ It is known that $I^r \subset I^{(s)}$ if and only if $r \ge s.$ However, the containment problem of symbolic powers in the ordinary powers of $I$ is a far more challenging problem. The containment problem aims to determine the positive integers $r$ and $s$ for which the inclusion $I^{(r)} \subset I^{s}$ holds.  Hochster and Huneke \cite{Hochster} establishes the containment $I^{(r(n-1))} \subset I^r$ for all $r \ge 1.$ Moreover, for certain classes of ideals, Harbourne established the improved containment $I^{(r(n-1)-n+2)} \subset I^{r}$ for all $r \ge 1.$ In an effort to further strengthen these containment relationships, Harbourne and Huneke \cite{Harbourne} introduced several conjectured containments in their work. In \cite{Cooper} Cooper et al.  used the symbolic polyhedron to obtain some containments of symbolic powers into ordinary powers of monomial ideals. These containment statements lead to bounds on the invariant $\al(I) = \min\{ d \mid I_d \neq 0 \}.$ Furthermore, for an ideal \( I \) of points in $\mathbb{P}^{n-1}$, various authors (see \cite{Chudnovsky1,Skoda,Waldschmidt}) obtained bounds on the invariant $\al(I)$. In Chudnovsky~\cite{Chudnovsky1} conjectured that
$$\frac{\al \left( I^{(s)} \right)}{s} \; \ge \; \dfrac{\al(I) + n - 2}{n-1},$$
where $I$ is an ideal of points in $\mathbb{P}^{n-1},$ where $ n \ge 2$ and $s \ge 1.$ It is known that the limit $\displaystyle \lim\limits_{s \to \infty}\frac{\al(I^{(s)})}{s}$ exists and is denoted by $\hat{\al}(I)$, which is referred as the \textit{Waldschmidt constant} of $I.$ Therefore, above conjecture can also be stated in terms of the Waldschmidt constant of $I.$

Let $I \subset K[x_{1},\ldots,x_{n}]$ be a monomial ideal with big height $h.$ When $h = n$, the maximal ideal $\mathfrak{m} = \langle x_{1},\ldots,x_{n}\rangle$ is an associated prime of $I$. Consequently, the symbolic and ordinary powers coincide, $I^{(s)} = I^{s}$ for all $s \ge 1$, (see \cite[Lemma 3.3]{Cooper}) and hence  $\hat{\al}(I)=\alpha(I).$ Therefore  $\hat{\al}(I)$ satisfies the lower bound of Chudnovsky's conjecture. Assuming that $h \le n - 1$ , Cooper et al.~\cite{Cooper} formulated an analogous conjecture in the context of monomial ideals.

\begin{conj} \label{C1}
    Let $I \subset K[x_1, \ldots, x_n]$ be a monomial ideal of big-height $h.$ Then 
    $$\hat{\al}(I) \ge \frac{\al( I) + h - 1}{h}.$$
\end{conj}    
    
Bocci et al. \cite{bocci2016waldschmidt} established Conjecture~\ref{C1} for the class of squarefree monomial ideals. Motivated by their work, the first objective of this paper is to examine Conjecture~\ref{C1} for the class of monomial ideals without embedded associated primes.

In Section~\ref{S-Waldschmidt}, we establish Conjecture~\ref{C1} for some classes of monomial ideals, beginning with the class of monomial ideals having a standard linear weighting (see Definition~\ref{D-Weighting}).

\vspace{3mm}
\noindent
    \textbf{Theorem \ref{Th-Waldschmidt}.}\textit{
    Let $I \subset R$ be a monomial ideal. If $I$ has a standard linear weighting and $h$ denotes the big-height of $I$, then Conjecture \ref{C1} holds for $I.$}
	%$$\hat{\al}(I) \ge \frac{\al(I) + h - 1}{h}.$$}

\vspace{3mm}

We now set our notation. Let $I \subset R$ be a monomial ideal without embedded associated primes, and let $I = \q_{1} \cap \cdots \cap \q_{r}$ be its minimal irreducible decomposition. Since for each $1 \le j \le r$, $\q_j$ is irreducible, we can write $\q_{j} = \langle x_{i}^{w_{j,i}} : i \in A_{j} \rangle$ for some $A_j \subset [n].$ The radical $\sqrt{I}$ of $I$ can be considered as the cover ideal of a hypergraph, which we denote by $\h_I.$ If $\h_I$ is a graph, then we denote it by $G_I.$ We prove Conjecture~\ref{C1} for the monomial ideal $I$ by assuming that $\h_I$ is a whiskered hypergraph (see Definition \ref{D-Whisker}) and applying certain conditions on the integers $w_{j,i}.$

\vspace{3mm}
\noindent
    \textbf{Theorem \ref{Th-Whisker}.}\textit{
   Let $I \subset R$ be a monomial ideal with no embedded associated primes and admitting the minimal irreducible decomposition. Suppose $\h_I$ is a whiskered hypergraph with whiskers $A_{j_1}, \dots, A_{j_t}$ attached to the vertices $i_1, \dots, i_t$, respectively. Assume that the following conditions hold:
    \begin{enumerate}[\rm(i)]
        \item $w_{j_k, i} \ge w_{j_k, i_k}$ for all $1 \le k \le t$ and $i \in A_{j_k}$, and
        %\item $V_1 = \{i_1, \dots, i_t\}$ with $V_1 \cap A_{j_k} = \{i_k\}$,
        \item for each $j \ne j_1, \dots, j_t$, there exists $1 \le k \le t$ with $i_k \in A_j$ such that $w_{j, i_k} \le w_{j_k, i_k}.$
    \end{enumerate}
    If $h$ is the big-height of $I$, then Conjecture \ref{C1} holds for $I.$}      

\vspace{3mm}

A monomial ideal $I \subset R$ is called a \textit{Simis ideal} if $I^{(s)} = I^{s}$ for all $s \ge 1.$ Providing a combinatorial characterization of Simis ideals remains a challenging open problem. In the case of squarefree monomial ideals, the problem coincides (see \cite{Conforti1}) with the classical packing conjecture of Conforti and Cornuejols. This problem has been resolved in several important cases (see \cite{BBKM,DasK,Gilter,GRV,GIV,HJTZ,MPD21,MPD,Simis} ). These results employ techniques from combinatorial optimization and graph theory. Méndez et al. in \cite{MVV} proposed the following conjecture.        

\begin{conj}\cite[Conjecture 5.7]{MVV}\label{C2}
    Let $I$ be a monomial ideal without any embedded associated primes. If the irreducible decomposition of $I$ is minimal and $I$ is a Simis ideal, then there exist a Simis squarefree monomial ideal $J$ and a standard linear weighting $w$ such that $I = J_{w}.$
\end{conj}    

This conjecture has been confirmed for the edge ideals of weighted oriented graphs \cite{GGMV}. Recently, the authors in \cite{bordoloi} established the conjecture for a class of support-$2$ monomial ideals. The second objective of this work is to study Conjecture \ref{C2} for certain classes of monomial ideals. These classes of monomial ideals are obtained by imposing specific conditions on the generating sets of their associated primes. In Section \ref{S-Simis}, we prove the following results.
    
\vspace{3mm}
\noindent
    \textbf{Theorem \ref{Th-General}.}\textit{
    Suppose that $A_j \not\subset A_k \cup A_l$ for every $j \in [r] \setminus \{k, l\}$ and for every pair $(k, l)$ satisfying 
    $w_{k,i} \ne w_{l,i} \text{ for some } i \in A_k \cap A_l.$
    Then Conjecture \ref{C2} holds for $I.$}
    
\vspace{3mm}
\noindent
    \textbf{Theorem \ref{Th-2Height}.}\textit{
    If $\height(\p) = 2$ for all $\p \in Ass(I)$, then Conjecture \ref{C2} holds for $I.$}

\vspace{3mm}

In \cite[Theorem 7.4]{MVV}, Méndez, Pinto, and Villarreal proved that if $D$ is a weighted oriented graph with underlying graph $G$, then $J(D)^{(2)} = J(D)^2$ if and only if every vertex in $V^{+}(D)$ is a sink and $G$ is bipartite (see \cite{MVV} for the definition of a weighted oriented graph and related terminology).
This inspires us to explore symbolic powers of a more general class of monomial ideals. 

\vspace{3mm}
\noindent
    \textbf{Corollary \ref{C-2Height}.}\textit{
    Let $I \subset R$ be a monomial ideal without embedded associated primes, and let $\height(\p) = 2$ for all $\p \in \Ass(I).$ Then the following conditions are equivalent:
    \begin{enumerate}[\rm(i)]
        \item $I$ is a Simis ideal.
        \item $I^{(2)} = I^2.$
        \item $G_I$ is a bipartite graph and $I$ has a standard linear weighting.
    \end{enumerate}}
    
\section{Preliminaries}\label{Prelim} 

Let $K$ be a field, and let $R$ denote the polynomial ring $K[x_1, \dots, x_n].$ 

\begin{defn}
	The \textit{support} of a monomial $f \in R$ is the set 
	$\supp(f) = \{i \in [n] : x_i \mid f\}.$ 
\end{defn}

\begin{defn}
	Let $I \subset R$ be an ideal.
	\begin{enumerate}[(a)]
		\item A prime ideal $\p$ is said to be an \textit{associated prime} of $I$ if there exists $f \in R$ such that $\ann(f + I) = \p$, where $\ann(f + I) = \{g \in R : fg \in I\}$ denotes the \textit{annihilator} of $f + I.$ The set of all associated prime ideals of $I$ is denoted by $\Ass(I).$
		
		\item An associated prime $\p$ of $I$ is called an \textit{embedded prime} of $I$ if there exists an associated prime $\p'$ of $I$ such that $\p' \subsetneq \p.$
		
		\item  The \textit{big-height} of an ideal $I \subset R$ is the maximum of the heights of its associated primes.
		
		\item A presentation of $I$ as intersection $I = \q_1 \cap \cdots \cap \q_r$, where $\q_1, \dots, \q_r$ are primary ideals, is called a \textit{primary decomposition} of $I.$ The primary decomposition is called \textit{minimal primary decomposition} if $\sqrt{\q_j} \ne \sqrt{\q_k}$ for all $j \ne k$ and $\bigcap_{k \ne j} \q_k \not\subset \q_j$ for all $1 \le j \le r.$
	\end{enumerate}
\end{defn}

\begin{rmk}
	Let $I = \q_1 \cap \cdots \cap \q_r$ be a minimal primary decomposition of $I$, and let $\sqrt{q_j} = \p_j$ for all $1 \le j \le r.$ Then $\Ass(I) = \{\p_1, \dots, \p_r\}.$
\end{rmk}

\begin{defn}
	An ideal $\q \subset R$ is called an \textit{irreducible ideal} if it cannot be written as the intersection of two ideals properly containing $\q.$ 
\end{defn}

\begin{rmk}
	Let $\q$ and $I$ be monomial ideals of $R.$
	\begin{enumerate}[(i)]
		\item By \cite[Corollary 1.3.2]{HBook}, $\q$ is irreducible if and only if $\q = \langle x_{i_1}^{w_1}, \dots, x_{i_a}^{w_a} \rangle$ for some $\{i_1, \dots, i_a\} \subset [n]$ and integers $w_1, \dots, w_a \in \mathbb{N}_{+}.$ In this case, $\q$ is $\langle x_{i_1}, \dots, x_{i_a} \rangle$-primary.
		
		\item In view of \cite[Theorem 1.3.1]{HBook}, $I$ can be written as intersection of irreducible monomial ideals. The representation of $I$ as intersection of irreducible monomial ideals is called a \textit{irreducible decomposition} of $I.$ Since irreducible ideals are primary ideals, the irreducible decomposition of $I$ is a primary decomposition of $I.$
	\end{enumerate}	 
\end{rmk}

We write $\mathbb{N}$ and $\mathbb{N}_{+}$, respectively, for the set of non-negative integers and positive integers.

\begin{defn}\label{D-Weighting}
    Let $w : \mathbb{R}^n \to \mathbb{R}^n$ be a linear function such that $w(\mathbb{N}_{+}^n) \subset \mathbb{N}_{+}^n.$
    \begin{enumerate}[(a)]
        \item The linear function $w$ is called a \textit{linear weighting} on $R.$

        \item For each $a = (a_1, \dots, a_n) \in \mathbb{N}^n$, let $x^a$ denote the monomial $x_1^{a_1} \cdots x_n^{a_n}$ of $R.$ A linear weighting $w$ of $R$ associates a new monomial $x^{w(a)}$ to the monomial $x^a.$ 
        
        \item The \textit{weighted monomial ideal} of a monomial ideal $I \subset R$, denoted $I_w$, is defined as
	    $$I_w = \langle x^{w(a)} : x^a \in I \rangle \subset R.$$

        \item A linear weighting $w$ on $R$ is said to be \textit{standard} if there exist positive integers $w_1, \dots, w_n$ such that $w(a) = (w_1 a_1, \dots, w_n a_n)$ for all $a = (a_1, \dots, a_n) \in \mathbb{R}^n.$ We write $w = (w_1, \dots, w_n)$ for the standard linear weighting $w.$

        \item We say that a monomial ideal $I \subset R$ \textit{has a standard linear weighting} if there exists a squarefree monomial ideal $J \subset R$ and a standard linear weighting $w = (w_1, \dots, w_n)$ on $R$ such that $I = J_w.$
    \end{enumerate}    
\end{defn}

\begin{defn}
	Let $I \subset R$ be a nonzero monomial ideal, and let $\al(I) = \min\{d : I_d \ne 0\}.$ The \textit{Waldschmidt constant} of $I$ is defined as
	$$\hat{\al}(I) = \lim_{s \to \infty} \frac{\al(I^{(s)})}{s}.$$
\end{defn}

\begin{defn}
	Let $I \subset R$ be an ideal. 
	\begin{enumerate}[(a)]
		\item The \textit{resurgence} of $I$ is defined by
		$$\rho(I) = \sup \left\{ \frac{s}{t} : s, t \in \mathbb{N}_{+} \text{ and } I^{(s)} \not\subset I^{t} \right\}.$$
		
		\item The \textit{asymptotic resurgence} of $I$ is defined by
		$$\rho_a(I) = \sup \left\{ \frac{s}{t} : s, t \in \mathbb{N}_{+} \text{ and } I^{(sp)} \not\subset I^{tp} \text{ for } p >> 0\right\}.$$
	\end{enumerate}	
\end{defn}

We end this section by introducing the notion of polarization.

\begin{defn}
	Let $I = \langle f_1, \dots, f_m \rangle \subset R$ be a monomial ideal. For each $1 \le j \le m$, assume that $f_j = x_1^{w_{j, 1}} \cdots x_n^{w_{j, n}}.$ Set $w_i = \max\{w_{j, i} : 1 \le j \le m\}$, where $1 \le i \le n.$ In the polynomial ring $T = K[x_{1, 1}, \dots, x_{1, w_1}, \dots, x_{n, 1},\dots,x_{n, w_n}]$, consider the squarefree monomials $\tilde{f}_1, \dots, \tilde{f}_m$, where 
	$$\tilde{f}_j = (x_{1, 1} \cdots x_{1, w_{j, 1}}) \cdots (x_{n, 1} \cdots x_{n, w_{j, n}}).$$
	The squarefree monomial ideal $\tilde{I} = \langle \tilde{f}_1, \dots, \tilde{f}_m \rangle \subset T$ is called the \textit{polarization} of $I.$
\end{defn}

\section{The Waldschmidt Constant of a Standard Weighted Monomial Ideal}\label{S-Waldschmidt}

In this section, we establish Conjecture~\ref{C1} for some classes of monomial ideals. At the beginning, we show that the resurgence and asymptotic resurgence are preserved under standard linear weighting and obtain an upper bound for the asymptotic resurgence of a monomial ideal having standard linear weighting.

\begin{defn}
	Let $\h$ be a hypergraph. Then the \textit{chromatic number} $\chi(\h)$ of $\h$ is defined to be the least number of colors required to color the vertices such that each edge has at least two vertices of different colors.
\end{defn} 

\begin{defn}
    Let $I \subset R$ be a monomial ideal. We define $\h_I$ to be the hypergraph on the vertex set $[n]$ whose cover ideal is $\sqrt{I}$. In particular, if $\h_I$ is a graph, we denote it by $G_I.$
\end{defn}

\begin{thm}
	Let $I \subset R$ be a monomial ideal, and let $w$ be a standard linear weighting on $R.$ Then $\rho(I_w) = \rho(I)$ and $\rho_a(I_w) = \rho_a(I).$ Moreover, if $h$ is the big-height of $I$ and $I$ has a standard linear weighting, then
    $$\rho_a(I) \le h - \frac{1}{\chi(\h_I)}.$$
\end{thm}

\begin{proof}
	Let $s, t, p \in \mathbb{N}_{+}.$ Then $I^{(sp)} \not\subset I^{tp}$ if and only if $(I^{(sp)})_w \not\subset (I^{tp})_w.$ In view of \cite[Lemma 1]{Toledo2021} and \cite[Lemma 4.3]{MVV}, this holds if and only if $(I_w)^{(sp)} \not\subset (I_w)^{tp}.$ This implies that $\rho_a(I_w) = \rho_a(I).$ By taking $p = 1$ in the above equivalent statements, we obtain  $\rho(I_w) = \rho(I).$ The last assersion follows from \cite[Theorem 3.12]{jayanthan2022resurgence}.
\end{proof}

\begin{rmk}
    Let $I \subset R$ be a monomial ideal of big-height $h.$ Assume that $I$ has a standard linear weighting $w.$ Then $I = J_w$ for some squarefree monomial ideal $J \subset R.$ Note that $\al(I) \ge \al(J)$ and $\hat{\al}(I) \ge \hat{\al}(J).$ Using \cite[Theorem 5.3]{bocci2016waldschmidt}, we obtain
    $$\hat{\al}(I) \ge \hat{\al}(J) \ge \frac{\al(J) + h - 1}{h}.$$
    One might expect that a lower bound for $\hat{\al}(J)$ could provide a lower bound of the same form for $\hat{\al}(I).$ However, since the inequality $\al(I) \ge \al(J)$ implies 
    $$\frac{\al(J) + h - 1}{h} \le \frac{\al(I) + h - 1}{h},$$
    the above estimates cannot be combined to deduce
    $$\hat{\al}(I) \ge \frac{\al(I) + h - 1}{h}.$$ 
    This shows that a lower bound for $\hat{\al}(J)$ does not automatically translate into a lower bound of the same form for $\hat{\al}(I).$
\end{rmk}

In the following, we fix the notation used throughout the rest of the paper.

\begin{notation}\label{notation}
    Let $s \in \mathbb{N}_{+}$, and let $I \subset R$ be a monomial ideal without embedded primes. Further, let 
    $$I = \q_1 \cap \cdots \cap \q_r$$ 
    be the minimal irreducible decomposition of $I.$ That is $\sqrt{\q_j} \ne \sqrt{\q_k}$ for all $j \ne k$ and each $\q_j$ is an irreducible monomial ideal. Thus, for each $1 \le j \le r$, we have $\q_j = \langle x_i^{w_{j,i}} : i \in A_j \rangle$ for some $A_j \subset [n].$ If $\p_j = \sqrt{\q_j} = \langle x_i : i \in A_j \rangle$, then $\Ass(I) = \{\p_1, \dots, \p_r\}.$ Note that the edge ideal of the hypergraph $\h_I$ is $\e(\h_I) = \{A_1, \dots, A_r\}.$ 
\end{notation}

\begin{rmk}\label{R-WeightedIdeal}
    Let $I \subset R$ be a monomial ideal defined as in  Notation~\ref{notation}. Then we observe the following:
    \begin{enumerate}[(i)]
		\item By using \cite[Lemma 2.6]{MVV}, the $s$-th symbolic power of $I$ is $I^{(s)} = \bigcap_{j = 1}^{r} \langle x_i^{w_{j,i}} : i \in A_j \rangle^s.$
		\item Using \cite[Proposition 2.5]{faridi}, the polarization of $I$ is given by
		$$\tilde{I} = \bigcap_{j = 1}^{r} \bigcap_{\substack{i \in A_j \\ 1 \le d_i \le w_{j,i}}} \langle x_{i, d_i} : i \in A_j \rangle.$$
		By applying \cite[Lemma 2.6]{MVV}, the $s$-th symbolic power of $\tilde{I}$ is
		$$\tilde{I}^{(s)} = \bigcap_{j = 1}^{r} \bigcap_{\substack{i \in A_j \\ 1 \le d_i \le w_{j,i}}} \langle x_{i, d_i} : i \in A_j \rangle^s.$$

        \item If $I$ has a standard linear weighting, then there exist a squarefree monomial ideal $J$ and a standard linear weighting $w = (w_1, \dots, w_n)$ on $R$ such that $I = J_w.$ It follows from \cite[Lemma 4.1]{MVV} that
        $I = J_w = \bigcap_{j = 1}^{r} (\langle x_i : i \in A_j \rangle)_w = \bigcap_{j = 1}^{r} \langle x_i^{w_i} : i \in A_j \rangle.$ 
        %In view of \cite[ Corollary 1.3.4]{HBook}, there exist $A_1, \dots A_r \subset [n]$ with $A_j \not \subset A_k$ for all $j \ne k$ such that $J = \bigcap_{j = 1}^{r} \langle x_i : i \in A_j \rangle.$
	\end{enumerate} 
\end{rmk}
 
\begin{lemma}\label{L-WaldschmidtC}
	With Notation~\ref{notation}, if $I$ has a standard linear weighting, then $\al(I^{(s)}) = \al(\tilde{I}^{(s)}).$
    %Consequently, $\hat{\al}(I) = \hat{\al}(\tilde{I}).$
\end{lemma}

\begin{proof}
	For a monomial ideal $L$, $\al(L)$ is the smallest degree of a non zero monomial in $L.$ Suppose $f \in I^{(s)}$ is a monomial with $\deg(f) = \al(I^{(s)}).$ Then $f$ is a monomial generator of $I^{(s)}.$ Using Remark \ref{R-WeightedIdeal}, for each $1 \le j \le r$, there exists a monomial generator $f_j$ of $\langle x_i^{w_i} : i \in A_j \rangle^s$ such that $f = \lcm(f_1, \dots, f_r).$ Let $1 \le j \le r$ and $i \in A_j.$ Then there exists $s_{j, i} \in \mathbb{N}$ with $\sum_{i \in A_j} s_{j, i} = s$ such that
	$$f_j = \prod_{i \in A_j} x_i^{w_i s_{j, i}}.$$
	If $s_i = \max\{s_{j, i} : i \in A_j\}$, then $f = \prod_{i \in [n]} x_i^{w_i s_i}.$ 
	Further, let
	$$\tilde{f}_j = \prod_{i \in A_j} (x_{i,1} \cdots x_{i,w_i})^{s_{j, i}}$$
	and $\tilde{f} = \lcm(\tilde{f}_1, \dots, \tilde{f}_r).$ Note that $\tilde{f} = \prod_{i \in [n]} (x_{i,1} \cdots x_{i,w_i})^{s_i}$ and $\deg(\tilde{f}) = \deg(f) = \al(I^{(s)}).$ Then, for each $i \in A_j$ and $1 \le d_i \le w_i$, the monomial $\prod_{i \in A_j} x_{i,d_i}^{s_{j, i}}$ divides $\tilde{f}_j.$ Since 
	$$\prod_{i \in A_j} x_{i,d_i}^{s_{j, i}} \in \langle x_{i,d_i} : i \in A_j \rangle^s,$$ 
	we get $\tilde{f}_j \in \langle x_{i,d_i} : i \in A_j \rangle^s$ for all $1 \le d_i \le w_i.$ This implies that $$\tilde{f}_j \in \bigcap_{\substack{1 \le d_i \le w_i \\ i \in A_j}} \langle x_{i,d_i} : i \in A_j \rangle^s.$$ 
	Thus, by Remark \ref{R-WeightedIdeal}, $\tilde{f} \in \tilde{I}^{(s)}$, and hence $\al(I^{(s)}) \ge \al(\tilde{I}^{(s)}).$
	
	To prove reverse inequality, suppose $\tilde{g} \in \tilde{I}^{(s)}$ is a monomial with $\deg(\tilde{g}) = \al(\tilde{I}^{(s)}).$ Then $\tilde{g}$ is a monomial generator of $\tilde{I}^{(s)}.$ For $1 \le j \le r$, let $i \in A_j$ and $1 \le d_i \le w_i.$ Then, by Remark~\ref{R-WeightedIdeal}, there exists $\tilde{g}_{j,d} \in \langle x_{i, d_i} : i \in A_j \rangle^s$ such that 
    $$\tilde{g} = \lcm\left( \tilde{g}_{j,d} : 1 \le j \le r \text{ and } 1 \le d_i \le w_i \text{ for all } i \in A_j \right),$$ 
    where $d$ is the $|A_j|$-tuple $(d_i : i \in A_j).$ Let $s_{j, i, d_i} \in \mathbb{N}$ be such that $\sum_{i \in A_j} s_{j, i, d_i} = s$ and
	$$\tilde{g}_{j,d} = \prod_{i \in A_j} x_{i,d_i}^{s_{j, i, d_i}}.$$
	If for each $i \in [n]$ and $1 \le p \le w_i$, we set $t_{i, p} = \max\{s_{j, i, d_i} : i \in A_j \text{ and } d_i = p\}$, then 
	$$\tilde{g} = \prod_{i \in [n]} x_{i,1}^{t_{i,1}} \cdots x_{i,w_i}^{t_{i, w_i}}.$$ 
	Let $g = \prod_{i \in [n]} x_i^{t_{i,1} + \cdots + t_{i,w_i}}.$ Then $\deg(g) = \deg(\tilde{g}) = \al(\tilde{I}^{(s)}).$ We prove that $g \in I^{(s)}.$ Choose $1 \le e_i \le w_i$ such that $t_{i,e_i} = \min\{t_{i, 1}, \dots, t_{i, w_i}\}.$ The monomial $g' = \prod_{i \in [n]} x_i^{w_i t_{i,e_i}}$ divides $g.$ Let $g_j$ denote the monomial $\prod_{i \in A_j} x_i^{w_i s_{j, i, e_i}}.$ One can check that $g_j \in \langle x_i^{w_i} : i \in A_j \rangle^s.$ Since $s_{j, i, e_i} \le t_{i,e_i}$, it follows that $g_j$ divides $g'.$ This implies that $g' \in \langle x_i^{w_i} : i \in A_j \rangle^s$, and hence by Remark \ref{R-WeightedIdeal}, $g \in I^{(s)}.$ Consequently, $\al(\tilde{I}^{(s)}) \ge \al(I^{(s)}).$
\end{proof}

In the following example, Lemma \ref{L-WaldschmidtC} does not hold if $I$ lacks a standard linear weighting.

\begin{exmp}
    Let $I = \langle x_1, x_2^2 \rangle \cap \langle x_2, x_3^2 \rangle \subset K[x_1, x_2, x_3].$ Then $I$ does not have a standard linear weighting. In view of Remark \ref{R-WeightedIdeal}, we see that the polarization of $I$ is 
    $$\tilde{I} = \langle x_{1,1}, x_{2,1} \rangle \cap \langle x_{1,1}, x_{2,2} \rangle \cap \langle x_{2,1}, x_{3,1} \rangle \cap \langle x_{2,1}, x_{3,2} \rangle.$$
    Consider the monomial $f = \lcm(x_1x_2^2, x_2^2) \in I^{(2)}.$ As in the proof of Lemma \ref{L-WaldschmidtC}, the monomial $\tilde{f} \in \tilde{I}^{(2)}$ is given by $\tilde{f} = \lcm(x_{1,1}x_{2,1}x_{2,2}, x_{2,1}^2).$ A computation on Macaulay2 \cite{M2} shows that $\al(I^{(2)}) = \deg(f) = 3$ and $\al(\tilde{I}^{(2)}) = \deg(\tilde{f}) = 4.$ Thus $\al(I^{(2)}) \ne \al(\tilde{I}^{(2)}).$
\end{exmp}

\begin{thm}\label{Th-Waldschmidt}
   Let $I \subset R$ be a monomial ideal. If $I$ has a standard linear weighting and $h$ denotes the big-height of $I$, then Conjecture \ref{C1} holds for $I.$
	%$$\hat{\al}(I) \ge \frac{\al(I) + h - 1}{h}.$$
\end{thm}

\begin{proof}
	By Lemma \ref{L-WaldschmidtC}, we have $\al(I^{(s)}) = \al(\tilde{I}^{(s)})$ for all $s \in \mathbb{N}_{+}$, and hence $\hat{\al}(I) = \hat{\al}(\tilde{I}).$ Note that the big-height of $\tilde{I}$ is the same as the big-height of $I.$ Using \cite[Theorem 5.3]{bocci2016waldschmidt}, we obtain
	$$\hat{\al}(I) = \hat{\al}(\tilde{I}) \ge \frac{\al(\tilde{I}) + h - 1}{h} = \frac{\al(I) + h - 1}{h}.$$
\end{proof}

\begin{defn}\label{D-Whisker}
    Let $\h$ be a hypergraph with vertex set $[n]$ and edge set $\e(\h) = \{A_1, \dots, A_r\}.$ The hypergraph $\h_I$ is said to be \textit{whiskered} if there exist mutually disjoint edges $A_{j_1}, \dots, A_{j_t}$ and a subset $V_1 = \{i_1, \dots, i_t\} \subset [n]$ such that
\begin{itemize}
    \item $A_j \subset V_1$ for all $j \ne j_1, \dots, j_t$; and

    \item $V_1 \cap A_{j_k} = \{i_k\}$ for all $1 \le k \le t.$  
\end{itemize}
We say that $\h_I$ is a whiskered hypergraph with whiskers $A_{j_1}, \dots, A_{j_t}$ attached to the vertices $i_1, \dots, i_t$, respectively.  %$A_{j_k}$ is the \textit{whisker} attached to the vertex $i_k.$ Observe that $V_1 = \{i_1, \dots, i_t\}.$ 
\end{defn}

\begin{exmp}
    Consider the hypergraph $\h$ with edge set
    $$\e(\h) = \left\{ \{1, 2, 3\}, \{2, 3, 4\}, \{1,5\}, \{2, 6, 7\}, \{3, 8\}, \{4, 9, 10\} \right\}.$$ If $A_{j_1} = \{1,5\}$, $A_{j_2} = \{2, 6, 7\}$, $A_{j_3} = \{3, 8\}$ and $A_{j_4} = \{4, 9, 10\}$, then $A_{j_1}, A_{j_2}, A_{j_3}$, $A_{j_4}$ are the whiskers attached on the vertices of $V_1 = \{1, 2, 3, 4\}.$ Thus, $\h$ is a whiskered hypergraph.
\end{exmp}

\begin{thm}\label{Th-Whisker}
  Let $I \subset R$ be a monomial ideal defined as in Notation \ref{notation}. Suppose $\h_I$ is a whiskered hypergraph with whiskers $A_{j_1}, \dots, A_{j_t}$ attached to the vertices $i_1, \dots, i_t$, respectively. Assume that the following conditions hold:
    \begin{enumerate}[\rm(i)]
        %\item $V_1 = \{i_1, \dots, i_t\}$ with $V_1 \cap A_{j_k} = \{i_k\}$,
        \item $w_{j_k, i} \ge w_{j_k, i_k}$ for all $1 \le k \le t$ and $i \in A_{j_k}$, and
        
        \item for each $j \ne j_1, \dots, j_t$, there exists $1 \le k \le t$ with $i_k \in A_j$ such that $w_{j, i_k} \le w_{j_k, i_k}.$
    \end{enumerate}
    If $h$ is the big-height of $I$, then Conjecture \ref{C1} holds for $I.$
\end{thm}

\begin{proof}
    Let $s \in \mathbb{N}_{+}$ and $\ell = s (w_{j_1, i_1} + \cdots + w_{j_t, i_t}).$ First we prove $\al(I^{(s)}) \ge \ell.$ Consider the monomial ideal $J = \q_{j_1} \cap \cdots \cap \q_{j_t}.$ By using \cite[Lemma 2.6]{MVV}, we get $J^{(s)} = \q_{j_1}^s \cap \cdots \cap \q_{j_t}^s.$ 
    If $g$ is a monomial generator of $J^{(s)}$, then for each $1 \le k \le t$, there exists $g_k = \prod_{i \in A_{j_k}} x_i^{w_{j_k, i} s_{k, i}} \in \q_{j_k}^s$ such that $g = \lcm(g_1, \dots, g_t)$, where for each $i \in A_{j_k}$, $s_{k, i} \in \mathbb{N}$ and $\sum_{i \in A_{j_k}} s_{k, i} = s.$ Since $A_{j_1}, \dots A_{j_t}$ are the whiskers of $\h_I$, we obtain that $A_{j_k} \cap A_{j_l} = \emptyset$ for all $k \ne l.$ This implies that $g = g_1 \cdots g_t$ and $\deg(g) = \sum_{k = 1}^{t} \sum_{i \in A_{j_k}} w_{j_k, i} s_{k, i}.$ Using condition (i), we obtain 
    $$\deg(g) \ge \sum_{k = 1}^{t} \sum_{i \in A_{j_k}} w_{j_k, i_k} s_{k, i} = \ell.$$
    This proves that $\al(J^{(s)}) \ge \ell.$ Since $I^{(s)} \subset J^{(s)}$, it follows that $\al(I^{(s)}) \ge \al(J^{(s)}) \ge \ell.$ Next, we prove that $\ell \ge \al(\tilde{I}^{(s)}).$ Let $\tilde{f} = \prod_{k = 1}^{t} (x_{i_k, 1} \cdots x_{i_k, w_{j_k, i_k}})^s.$ Then $\deg(\tilde{f}) = \ell.$ We claim that $\tilde{f} \in \tilde{I}^{(s)}.$ Let $1 \le j \le r$, and for each $i \in A_j$, let $1 \le d_i \le w_{j, i}.$ In view of Remark \ref{R-WeightedIdeal}, it suffices to show that $\tilde{f} \in  \langle x_{i, d_i} : i \in A_j \rangle^s.$ We consider the following possibilities.
    \begin{itemize}
        \item Assume that $j = j_k$ for some $1 \le k \le t.$ Observe that $x_{i_k, d_{i_k}}^s \in  \langle x_{i, d_i} : i \in A_j \rangle^s.$ Since $x_{i_k, d_{i_k}}^s \mid \tilde{f}$, it follows that $\tilde{f} \in  \langle x_{i, d_i} : i \in A_j \rangle^s.$

        \vspace{3mm}
        \item Assume that $j \ne j_k$ for all $1 \le k \le t.$ By using condition (ii), there exists $1 \le k \le t$ with $i_k \in A_j$ such that $w_{j, i_k} \le w_{j_k, i_k}.$ Thus $d_{i_k} \le w_{j_k, i_k}$ and $x_{i_k, d_{i_k}}^s \in  \langle x_{i, d_i} : i \in A_j \rangle^s.$ This implies that $x_{i_k, d_{i_k}}^s \mid \tilde{f}$, which further implies  $\tilde{f} \in  \langle x_{i, d_i} : i \in A_j \rangle^s.$
    \end{itemize}
  
    Hence $\tilde{f} \in \tilde{I}^{(s)}.$ This proves that $\ell = \deg(\tilde{f}) \ge \al(\tilde{I}^{(s)}).$ By combining our results, we obtain that $\al(I^{(s)}) \ge \al(\tilde{I}^{(s)}).$ Consequently, $\hat{\al}(I) \ge \hat{\al}(\tilde{I}).$ One can check that $\al(I) = \al(\tilde{I}).$ Now, it follows from \cite[Theorem 5.3]{bocci2016waldschmidt} that
	$$\hat{\al}(I) \ge \hat{\al}(\tilde{I}) \ge \frac{\al(\tilde{I}) + h - 1}{h} = \frac{\al(I) + h - 1}{h}.$$
\end{proof}

\section{The Mendez-Pinto-Villarreal Conjecture}\label{S-Simis}

In this section, we prove that Conjecture \ref{C2} holds for some classes of monomial ideals. We begin by fixing the notation and basic setup for this section. For a matrix of integers $M$, we write $M^{\top}$ to denote its transpose. For $a, b \in \mathbb{N}^p$, we write $a \le b$ if $a_i \le b_i$ for all $1 \le i \le p.$

\begin{set}\label{setup}
    Let $I \subset R$ be a monomial ideal without embedded primes.
	With Notation~\ref{notation}, the minimal irreducible decomposition of $I$ is $I = \q_1 \cap \cdots \cap \q_r$, where for each $1 \le j \le r$, we have $\q_j = \langle x_i^{w_{j,i}} : i \in A_j \rangle$,  $\p_j = \sqrt{\q_j}$, and $A_j \subset [n].$ Also, we have $\Ass(I) = \{\p_1, \dots, \p_r\}.$ Without loss of generality, we may assume that $[n] = A_1 \cup \cdots \cup A_r.$ Let $X_I$ denote the set 
	$$\{(k, l) : 1 \le k < l \le r, \text{ there exists } i \in A_k \cap A_l \text{ with } w_{k, i} \ne w_{l, i}\}.$$ 
	For convenience, we write $A = A_1 \times \cdots \times A_r$, and for $i \in [n]$, we define 
	$$B_i = \{a \in A : \pi_j(a) = i \text{ for some } 1 \le j \le r\},$$
	where $\pi_j : A \to A_j$ is the $j$-th projection map. If $a \in A$, then we write $\pi_j(a) = a_j.$ Thus, we have $a = (a_1, \dots, a_r).$ By fixing an order on the elements of the set $A$, we regard $\mathbb{N}^{p \times |A|}$ as the set of all $p \times |A|$ matrices with non-negative entries, whose columns are indexed by the elements of $A.$ We associate to $I$ a matrix $P_I \in \mathbb{N}^{n \times |A|}$ defined as $P_I = [c_{i,a}]$, where 
	$$c_{i,a} = \max\{w_{j,a_j} : 1 \le j \le r,\; i = a_j\}$$ 
	for $a \in B_i$, and $c_{i,a} = 0$ for $a \notin B_i.$
\end{set}

\begin{rmk}\label{R-SLW}
     With the notation of Set-up \ref{setup}, if $X_I = \emptyset$, then $w_{j,i} = w_{k,i}$ for all $j \ne k.$ For each $i \in [n]$, let $w_i = w_{j,i}$, where $1 \le j \le r$ with $i \in A_j.$ Then $w = (w_1, \dots, w_n)$ is a standard linear weighting on $R.$ Also, if $J = \sqrt{I}$, then $I = J_w$ (see Remark \ref{R-WeightedIdeal}).
\end{rmk}

\begin{lemma}\label{L-MatrixSystem}
	With the notation of Set-up \ref{setup}, a monomial $f = x_1^{\al_1} \cdots x_n^{\al_n} \in I^t$ if and only if there exists $[m_a] \in \mathbb{N}^{1 \times |A|}$ such that $\sum_{a \in A} m_a = t$ and 
	$$P_I [m_a]^{\top} \le [\al_1 \cdots \al_n]^{\top}.$$ 
\end{lemma}

\begin{proof}
	Notice that
	$$I = \langle \lcm \left(x_{a_1}^{w_{1, a_1}}, \dots, x_{a_r}^{w_{r, a_r}} \right) : (a_1, \dots, a_r) \in A \rangle.$$
	First suppose that $f \in I^t.$ Then, for each $a = (a_1, \dots, a_r) \in A$, there exists an integer $m_a \ge 0$ with $\sum_{a \in A} m_a = t$ such that the monomial
	$$g = \prod_{a \in A} \left( \lcm 
	\left(x_{a_1}^{w_{1, a_1}}, \dots, x_{a_r}^{w_{r, a_r}} \right) \right)^{m_a}$$
	divides $f.$ One can observe that $g = x_1^{\beta_1} \cdots x_n^{\beta_n}$, where $\beta_i = \sum_{a \in A} c_{i, a} m_a.$ Since $g \mid f$, it follows that $\beta_i = \sum_{a \in A} c_{i, a} m_a \le \al_i$ for all $1 \le i \le n.$ This implies that $P_I [m_a]^{\top} \le [\al_1 \cdots \al_n]^{\top}.$
	
	Conversely, suppose that there exists $[m_a] \in \mathbb{N}^{1 \times |A|}$ such that $\sum_{a \in A} m_a = t$ and 
	$$P_I [m_a]^{\top} \le [\al_1 \cdots \al_n]^{\top}.$$ 
    If $\beta_i = \sum_{a \in A} c_{i, a} m_a$, then the monomial $g = x_1^{\beta_1} \cdots x_n^{\beta_n}$ divide $f$ and $g \in I^t.$ Thus $f \in I^t.$
\end{proof}

\begin{lemma}\label{L-General}
	With the notation of Set-up \ref{setup}, if $X_I \ne \emptyset$ and for each $(k, l) \in X_I$, it holds that $A_j \not\subset A_k \cup A_l$ for every $j \in [r] \setminus \{k, l\}$, then $I$ is not a Simis ideal. 
\end{lemma}

\begin{proof}
	Choose $(k, l) \in X_I$ and $i_0 \in A_k \cap A_l.$ If we set $w_1 = \min\{w_{j, i_0} : 1 \le j \le r,\; i_0 \in A_j\}$ and $w_2 = \max\{w_{j, i_0} : 1 \le j \le r,\; i_0 \in A_j\}$, then $w_1 < w_2.$ By relabeling, we can choose $1 \le t < s \le r$ such that
	\begin{itemize}
		\item $i_0 \in A_1 \cap \cdots \cap A_s$; and 
		
		\item $w_{j, i_0} = w_1$ for all $1 \le j \le t$, $w_{j, i_0} > w_1$ for all $t + 1 \le j \le s$ and $w_{s, i_0} = w_2.$
	\end{itemize}
	Note that $(1, s) \in X_I.$ For each $t + 1 \le j \le r$ with $j \ne s$, let $i_j \in A_j \setminus (A_1 \cup A_s)$ and $i_s \in A_s \setminus A_1.$ By the division algorithm, we write $w_2 = \al w_1 + \beta$, where $0 \le \beta < w_1.$ We consider the monomial $f =  \lcm \left( f_1, \dots, f_r \right)$, where for each $1 \le j \le r$, 
	$$f_j = \begin{cases}
		x_{i_0}^{(\al + 1) w_1} & \text{ if } 1 \le j \le t;  \\ 
		\vspace{-2mm} \\
		x_{i_0}^{w_{j, i_0}} x_{i_j}^{\al w_{j, i_j}} &  \text{ if } t + 1 \le j \le s; \\
		\vspace{-2mm} \\
		x_{i_j}^{(\al + 1) w_{j, i_j}} &  \text{ if } s < j \le r.	
	\end{cases}$$
	Then $f \in I^{(\al + 1)}$ and $\supp(f) = \{i_0, i_{t + 1}, \dots, i_r\}.$ If $f = x_1^{\gamma_1} \cdots x_n^{\gamma_n}$, then $\gamma_{i_0} = (\al + 1) w_1$ and $\gamma_{i_s} = \al w_{s, i_s}.$ We prove that $f \notin I^{\al + 1}.$ To get a contradiction, suppose that $f \in I^{\al + 1}.$ Then, by Lemma \ref{L-MatrixSystem}, there exists $[m_a] \in \mathbb{N}^{1 \times |A|}$ such that $\sum_{a \in A} m_a = \al + 1$ and 
	$$P_I [m_a]^{\top} \le [\al_1 \cdots \al_n]^{\top}.$$
	This implies that
	\begin{equation}\label{INQ, 1}
		\sum_{a \in B_{i_0}} m_a c_{i_0, a} \le (\al + 1) w_1
	\end{equation}
	and
	\begin{equation}\label{INQ, t+1}
		\sum_{a \in B_{i_s}} m_a c_{i_s, a} \le \al w_{s, i_s}.
	\end{equation}
	Let $B = \{a \in A : a_j \in \supp(f) \text{ for all } j \in [r]\}.$ Then, we observe that $m_a = 0$ for all $a \in A \setminus B.$ By using $\sum_{a \in A} m_a = \al + 1$, we obtain $\sum_{a \in B} m_a = \al + 1.$ If $a \in B$, then $a_1 \in \supp(f).$ Since $i_j \notin A_1$ for all $t + 1 \le j \le r$, it follows that $a_1 = i_0.$ This implies that $a \in B_{i_0}$, and hence $B \subset B_{i_0}.$ If we set $P = \{a \in B : c_{i_0, a} = w_1\}$ and $Q = \{a \in B : c_{i_0, a} > w_1\}$, then $B = P \cup Q.$
	In light of Inequality \eqref{INQ, 1}, we see that
	\begin{align*}
		(\al + 1) w_1 & \ge \sum_{a \in B} m_a c_{i_0, a} + \sum_{a \in B_{i_0} \setminus B} m_a c_{i_0, a} \\
		& \ge \sum_{a \in P} m_a w_1 + \sum_{a \in Q} m_a (w_1 + 1) \\
		& = \sum_{a \in B} m_a w_1 + \sum_{a \in Q} m_a \\
		& = (\al + 1) w_1 + \sum_{a \in Q} m_a.
	\end{align*}
	This implies that $m_a = 0$ for all $a \in Q$, which further implies that $\sum_{a \in P} m_a = \al + 1.$ Suppose that $a \in P.$ Then $a_s \in \supp(f)$ and $c_{i_0, a} = w_1.$ Since $w_{s, i_0} = w_2 > w_1$, it follows that $a_s \ne i_0.$ Again, since $i_j \notin A_s$ for all $t + 1 \le j \le r$ with $j \ne s$, we get $a_s = i_s.$ Hence $c_{i_s, a} \ge w_{s, i_s}$ and $P \subset B_{i_s}.$ By using Inequality \eqref{INQ, t+1}, we obtain 
	$$\al w_{s, i_s} \ge \sum_{a \in P} m_a w_{s, i_s} + \sum_{a \in B_{i_s} \setminus P} m_a c_{i_s, a} \ge (\al + 1) w_{s, i_s}.$$
	This yields a contradiction.
\end{proof}

\begin{lemma}\label{L-General2}
	With the notation of Set-up~\ref{setup}, let $\height(\p) = 2$ for all $\p \in Ass(I).$ If $X_I \ne \emptyset$ and for each $(k, l) \in X_I$, it holds that $A_j \not\subset A_k \cup A_l$ for every $j \in [r] \setminus \{k, l\}$, then $I^{(2)} \not\subset I^2.$
\end{lemma}

\begin{proof}
	Choose $(k, l) \in X_I$ and $i_0 \in A_k \cap A_l.$ If we set $w_1 = \min\{w_{j, i_0} : 1 \le j \le r,\; i_0 \in A_j\}$ and $w_2 = \max\{w_{j, i_0} : 1 \le j \le r,\; i_0 \in A_j\}$, then $w_1 < w_2.$ By relabeling, we can choose $1 \le t < s \le r$ such that
	\begin{itemize}
		\item $i_0 \in A_1 \cap \cdots \cap A_s$; and 
		
		\item $w_{j, i_0} = w_1$ for all $1 \le j \le t$, $w_{j, i_0} > w_1$ for all $t + 1 \le j \le s$ and $w_{s, i_0} = w_2.$
	\end{itemize}
	Since $(1, s) \in X_I$, for each $2 \le j \le r$ with $j \ne s$, we have $A_j \not\subset A_1 \cup A_s.$ For each $2 \le j \le r$ with $j \ne s$, let $i_j \in A_j \setminus (A_1 \cup A_s)$, $i_1 \in A_1 \setminus A_s$ and $i_s \in A_s \setminus A_1.$ We consider the following two cases. 

    \vspace{2mm}
	\noindent
	\textbf{Case 1.} When $2w_1 \le w_2.$ In this case, we set 
    $$f = x_{i_0}^{w_2} x_{i_s}^{w_{s, i_s}} \prod_{\substack{j = 2 \\ j \ne s}}^{r} x_{i_j}^{2w_{j,i_j}}.$$
    First we prove $f \in I^{(2)}.$ Since $I^{(2)} = \bigcap_{j = 1}^{r} \q_j^2$, it suffices to show that $f \in \q^2$ for all $1 \le j \le r.$ We consider the following possibilities.
    \begin{itemize}
        \item If $j = 1$, then $x_{i_0}^{2w_1} \in \q_1^2.$ Since $2w_1 \le w_2$, we get $x_{i_0}^{2w_1}$ divides $f.$ Thus $f \in \q_1^2.$ 

        \vspace{2mm}
        \item If $j = s$, then $x_{i_0}^{w_2} x_{i_s}^{w_{s, i_s}} \in \q_s^2.$ Since $x_{i_0}^{w_2} x_{i_s}^{w_{s, i_s}}$ divides $f$, we get $f \in \q_s^2.$ 
        
        \vspace{2mm}
        \item If $j \ne 1, s$, then $x_{i_j}^{2w_{j,i_j}} \in \q_j^2$ and $x_{i_j}^{2w_{j,i_j}}$ divides $f.$ Thus $f \in \q_j^2.$
    \end{itemize}
    Hence $f \in I^{(2)}.$ Next we prove that $f \notin I^2.$ Note that $I \subset \q_1 \cap \q_s$, and hence $I^2 \subset (\q_1 \cap \q_s)^2.$ Since $\height{\p_1} = \height{\p_s} = 2$, we can write $\q_1 = \langle x_{i_0}^{w_1}, x_{i_1}^{w_{1, i_1}} \rangle$ and $\q_s = \langle x_{i_0}^{w_2}, x_{i_s}^{w_{s, i_s}} \rangle.$ This implies that
    \begin{align*}
        (\q_1 \cap \q_s)^2 = & \left( \left\langle x_{i_0}^{w_2}, x_{i_0}^{w_1} x_{i_s}^{w_{s, i_s}}, x_{i_1}^{w_{1, i_1}} x_{i_s}^{w_{s, i_s}} \right\rangle \right)^2 \\
        = & \left\langle x_{i_0}^{2w_2}, x_{i_0}^{w_1 + w_2} x_{i_s}^{w_{s, i_s}}, x_{i_0}^{w_2} x_{i_1}^{w_{1, i_1}} x_{i_s}^{w_{s, i_s}}, x_{i_0}^{2w_1} x_{i_s}^{2w_{s, i_s}}, x_{i_0}^{w_1} x_{i_1}^{w_{1, i_1}} x_{i_s}^{2w_{s, i_s}}, x_{i_1}^{2w_{1, i_1}} x_{i_s}^{2w_{s, i_s}} \right\rangle.
    \end{align*}
    Therefore, $f \notin (\q_1 \cap \q_s)^2$, and hence $f \notin I^2.$

    \vspace{2mm}
	\noindent
	\textbf{Case 2.} When $2w_1 > w_2$. In this case, by the division algorithm, we can write $w_2 = w_1 + \beta$, where $0 \le \beta < w_1.$  Now, one can simply follow the same steps as in the proof of Lemma \ref{L-General} with $\alpha = 1$ to show that $I^{(2)} \not\subset I^2.$ This completes the proof.
\end{proof}

\begin{lemma}\label{L-2Height}
	With the notation of Set-up \ref{setup}, let $\height(\p) = 2$ for all $\p \in Ass(I).$ If $X_I \ne \emptyset$, then $I^{(2)} \not \subset I^2.$
\end{lemma}

\begin{proof}
	If for each $(k, l) \in X_I$, it holds that $A_j \not\subset A_k \cup A_l$ for every $j \in [r] \setminus \{k, l\}$, then the result follows from Lemma \ref{L-General2}. Next, we assume that there exist $(k, l) \in X_I$ and $p \in [r] \setminus \{k, l\}$ such that $A_p \subset A_k \cup A_l.$ Since $(k, l) \in X_I$, there exists $i_0 \in A_k \cap A_l$ with $w_{k, i_0} \ne w_{l, i_0}.$ Let $w_1 = \min\{w_{j,i_0} : i_0 \in A_j\}$ and $w_2 = \max\{w_{j,i_0} : i_0 \in A_j\}.$ By relabeling, we can choose $1 < s \le r$ such that
	\begin{itemize}
		\item $i_0 \in A_1 \cap \cdots \cap A_s$; and 
		
		\item $w_1 = w_{1, i_0} \le \cdots \le w_{k, i_0} < w_{l, i_0} \le \cdots \le w_{s, i_0} = w_2.$
	\end{itemize}
	For each $1 \le j \le s$, write $A_j = \{i_0, i_j\}.$ Then we must have $A_p = \{i_k, i_l\}.$ Now, we consider the following two cases.
    
	\vspace{2mm}
	\noindent
	\textbf{Case 1.} When $w_1 \le w_{k, i_0} < w_{l, i_0} = w_2.$ In this case, we have $A_j \not\subset \{i_0, i_k, i_l\}$, where $s + 1 \le j \le r$ and $j \ne p.$ For each $1 \le j \le r$, let
	$$f_j = \begin{cases}
		x_{i_0}^{2w_{j, i_0}} & \text{ if } 1 \le j \le k;  \\ 
		\vspace{-2mm} \\
		x_{i_0}^{w_{j, i_0}} x_{i_j}^{w_{j, i_j}} &  \text{ if } k + 1 \le j \le s; \\
		\vspace{-2mm} \\
		x_{i_k}^{w_{p, i_k}} x_{i_l}^{w_{p, i_l}} &  \text{ if } j = p; \\
		\vspace{-2mm} \\
		x_{i_j}^{2w_{j, i_j}} &  \text{ if } s + 1 \le j \le r, ~ j \ne p \text{ and } i_j \in A_j \setminus \{i_0, i_k, i_l\}.	
	\end{cases}$$
	Then the monomial 
	$$f =  \lcm \left( f_1, \dots, f_r \right) \in I^{(2)}.$$ 
	If $f = x_1^{\al_1} \cdots x_n^{\al_n}$, then $\al_{i_k} = w_{p, i_k}$, $\al_{i_0} = \max\{2w_{k, i_0}, w_2\}$ and $\al_{i_l} = \max\{w_{l, i_l}, w_{p, i_l}\}.$ We prove that $f \notin I^2.$ To get a contradiction, suppose that $f \in I^2.$ Then, by Lemma \ref{L-MatrixSystem}, there exists $[m_a] \in \mathbb{N}^{1 \times A}$ such that
	\begin{equation}\label{EQN2}
		\sum_{a \in A} m_a = 2
	\end{equation}
	and 
	$$P_I [m_a]^{\top} \le [\al_1 \cdots \al_n]^{\top}.$$
	This implies that
	\begin{equation}\label{INQ, i_0}
		\sum_{a \in B_{i_0}} m_a c_{i_0, a} \le \max\{2w_{k, i_0}, w_2\},
	\end{equation}
	
	\begin{equation}\label{INQ, i_k}
		\sum_{a \in B_{i_k}} m_a c_{i_k, a} \le w_{p, i_k},
	\end{equation}
	and
	\begin{equation}\label{INQ, i_l}
		\sum_{a \in B_{i_l}} m_a c_{i_l, a} \le \max\{w_{l, i_l}, w_{p, i_l}\}.
	\end{equation}
	We consider the following two subcases.
    
	\vspace{2mm}
	\noindent
	\textbf{Subcase 1}(a). When $w_{l, i_l} = \max\{w_{l, i_l}, w_{p, i_l}\}.$ Consider the sets $C_1 = \{a \in B_{i_l} : c_{i_l, a} \ge w_{l, i_l}\}$ and $C_2 = \{a \in B_{i_l} : c_{i_l, a} < w_{l, i_l}\}.$ Then we have $B_{i_l} = C_1 \cup C_2.$ Let $a \in A \setminus C_1.$ If $a_l = i_l$, then $a \in B_{i_l}$ and $c_{i_l,a} \ge w_{l,i_l}$, which implies $a \in C_1$, a contradiction. Therefore, $a_l = i_0$, and hence $a \in B_{i_0}$ and $c_{i_0,a} = w_2$. Consequently, $A \setminus C_1 \subset B_{i_0}.$ By using Inequality \eqref{INQ, i_l}, we obtain 
	\begin{equation*}\label{INQ8}
		\sum_{a \in C_1} m_a w_{l, i_l} \le \sum_{a \in C_1} m_a c_{i_l, a} + \sum_{a \in C_2} m_a c_{i_l, a} \le w_{l, i_l}.
	\end{equation*}
	This gives $\sum_{a \in C_1} m_a \le 1.$ If $\sum_{a \in C_1} m_a = 0$, then by Equation \eqref{EQN2}, we get $\sum_{a \in A \setminus C_1} m_a = 2.$ In light of Inequality \eqref{INQ, i_0}, we obtain
	$$\max\{2w_{k, i_0}, w_2\} \ge \sum_{a \in A \setminus C_1} m_a c_{i_0, a} = 2w_2.$$
	This contradicts the inequality $0 < w_{k, i_0} < w_2.$ Hence $\sum_{a \in C_1} m_a = 1.$ By Equation \eqref{EQN2}, we have $\sum_{a \in A \setminus C_1} m_a = 1.$ Since $a_k \in A_k = \{i_0, i_k\}$, we can write
	$$\sum_{a \in C_1, a_k = i_0} m_a + \sum_{a \in C_1, a_k = i_k} m_a = 1.$$
	We prove that $\sum_{a \in C_1, a_k = i_k} m_a = 0.$ Using Inequality \eqref{INQ, i_l}, we obtain
	$$w_{l, i_l} \ge \sum_{a \in C_1} m_a c_{i_l, a} + \sum_{a \in C_2} m_a c_{i_l, a} \ge w_{l, i_l} + \sum_{a \in C_2} m_a c_{i_l, a}.$$
	Thus, we have $\sum_{a \in C_2} m_a c_{i_l, a} = 0$, and hence $m_a = 0$ for all $a \in C_2.$ Since $B_{i_l} = C_1 \cup C_2$, we get $A \setminus C_1 = C_2 \cup (A \setminus B_{i_l}).$ Consequently, we obtain $\sum_{a \in A \setminus B_{i_l}} m_a = \sum_{a \in A \setminus C_1} m_a = 1.$ For each $a \in A \setminus B_{i_l}$, we have $a_p \ne i_l.$ Since $a_p \in A_p = \{i_k, i_l\}$, it follows that $a_p = i_k$, and hence $a \in B_{i_k}.$ This proves that $A \setminus B_{i_l} \subset B_{i_k}$ and $c_{i_k, a} \ge w_{p, i_k}$ for all $a \in A \setminus B_{i_l}.$ Also, we have $\{a \in C_1 : a_k = i_k\} \subset B_{i_k}.$ Further, since $C_1 \subset B_{i_l}$, we obtain 
	$$\{a \in C_1 : a_k = i_k\} \cap (A \setminus B_{i_l}) = \emptyset.$$ 
	Using Inequality \eqref{INQ, i_k}, we find that
	$$w_{p, i_k} \ge \sum_{a \in A \setminus B_{i_l}} m_a c_{i_k, a} + \sum_{a \in C_1, a_k = i_k} m_a c_{i_k, a} \ge w_{p, i_k} + \sum_{a \in C_1, a_k = i_k} m_a c_{i_k, a}.$$
	This proves that $\sum_{a \in C_1, a_k = i_k} m_a = 0.$ Consequently, $\sum_{a \in C_1, a_k = i_0} m_a = 1.$ Since $A \setminus C_1 \subset B_{i_0}$ and $\{a \in A : a_k = i_0\} \subset B_{i_0}$, it follows that
	\begin{align*}
		B_{i_0} = B_{i_0} \cap A = B_{i_0} \cap ((A \setminus C_1) \cup C_1) & = (A \setminus C_1) \cup (B_{i_0} \cap C_1) \\
		& \supset (A \setminus C_1) \cup (\{a \in A : a_k = i_0\} \cap C_1).
	\end{align*}
	Thus, by Inequality \eqref{INQ, i_0}, we obtain
	$$\max\{2w_{k, i_0}, w_2\} \ge \sum_{a \in A \setminus C_1} m_a c_{i_0, a} + \sum_{a \in C_1, a_k = i_0} m_a c_{i_0, a}.$$
	Note that $c_{i_0, a} \ge w_{k, i_0}$ for all $a \in C_1$ with $a_k = i_0.$ Thus, $\max\{2w_{k,i_0},w_2\} \ge w_2 + w_{k,i_0}.$ This contradicts the inequality $0 < w_{k, i_0} < w_2.$

    \vspace{2mm}
	\noindent
	\textbf{Subcase 1}(b). When $w_{p, i_l} = \max\{w_{l, i_l}, w_{p, i_l}\}.$ Consider the sets $D_1 = \{a \in B_{i_l} : c_{i_l, a} \ge w_{p, i_l}\}$ and $D_2 = \{a \in B_{i_l} : c_{i_l, a} < w_{p, i_l}\}.$ Then we have $B_{i_l} = D_1 \cup D_2.$ Let $a \in A \setminus D_1.$ If $a_p = i_l$, then $a \in B_{i_l}$ and $c_{i_l, a} \ge w_{p, i_l}$, which implies $a \in D_1$, a contradiction. Therefore, $a_p = i_k$, and hence $a \in B_{i_k}$ and $c_{i_k, a} \ge w_{p, i_k}.$ Consequently, $A \setminus D_1 \subset B_{i_k}.$
	By using Inequality \eqref{INQ, i_l}, we obtain 
	\begin{equation*}\label{INQ8}
		\sum_{a \in D_1} m_a w_{p, i_l} \le \sum_{a \in D_1} m_a c_{i_l, a} + \sum_{a \in D_2} m_a c_{i_l, a} \le w_{p, i_l}.
	\end{equation*} 
	This gives $\sum_{a \in D_1} m_a \le 1.$ If $\sum_{a \in D_1} m_a = 0$, then by Equation \eqref{EQN2}, we have $\sum_{a \in A \setminus D_1} m_a = 2.$ In light of Inequality \eqref{INQ, i_k}, we obtain
	$w_{p, i_k} \ge \sum_{a \in A \setminus D_1} m_a c_{i_k, a} \ge 2w_{p, i_k}$, a contradiction. Hence $\sum_{a \in D_1} m_a = 1.$ By Equation \eqref{EQN2}, we have $\sum_{a \in A \setminus D_1} m_a = 1.$ Since $A \setminus D_1 \subset B_{i_k}$, it follows that
	$$B_{i_k} = B_{i_k} \cap A = B_{i_k} \cap ((A \setminus D_1) \cup D_1) = (A \setminus D_1) \cup (B_{i_k} \cap D_1).$$
	Thus, by Inequality \eqref{INQ, i_k}, we obtain
	$$w_{p, i_k} \ge \sum_{a \in A \setminus D_1} m_a c_{i_k, a} + \sum_{a \in B_{i_k} \cap D_1} m_a c_{i_k, a} \ge w_{p, i_k} + \sum_{a \in B_{i_k} \cap D_1} m_a c_{i_k, a}.$$
	This implies that $\sum_{a \in B_{i_k} \cap D_1} m_a c_{i_k, a} = 0$, and hence $m_a = 0$ for all $a \in B_{i_k} \cap D_1.$ Using Inequality \eqref{INQ, i_l}, we obtain
	$$w_{p, i_l} \ge \sum_{a \in D_1} m_a c_{i_l, a} + \sum_{a \in D_2} m_a c_{i_l, a} \ge w_{p, i_l} + \sum_{a \in D_2} m_a c_{i_l, a}.$$
	This proves $\sum_{a \in D_2} m_a c_{i_l, a} = 0$, which implies $m_a = 0$ for all $a \in D_2.$ Since $B_{i_l} = D_1 \cup D_2$, we get $A \setminus D_1 = D_2 \cup (A \setminus B_{i_l}).$ Consequently, we obtain $\sum_{a \in A \setminus B_{i_l}} m_a = \sum_{a \in A \setminus D_1} m_a = 1.$ If $a \in D_1 \setminus B_{i_0}$, then $a_k \ne i_0.$ Since $a_k \in A_k = \{i_0, i_k\}$, we get $a_k = i_k$, and hence $a \in B_{i_k} \cap D_1.$ Thus, $m_a = 0$ for all $a \in D_1 \setminus B_{i_0}.$ By decomposing $D_1$ as
	$$D_1 = (B_{i_0} \cap D_1) \cup (D_1 \setminus B_{i_0}),$$
	we obtain $\sum_{a \in B_{i_0} \cap D_1} m_a = \sum_{a \in D_1} m_a = 1.$ If $a \in A \setminus B_{i_l}$, then $a_l \ne i_l.$ Since $a_l \in A_l = \{i_0, i_l\}$, we get $a_l = i_0$, and hence $a \in B_{i_0}$ and $c_{i_0, a} = w_2.$ This implies that $(A \setminus B_{i_l}) \cup (B_{i_0} \cap D_1) \subset B_{i_0}.$ Since $D_1 \subset B_{i_l}$, it follows that
	$(A \setminus B_{i_l}) \cap (B_{i_0} \cap D_1) \subset B_{i_0}.$ Thus, by Inequality \eqref{INQ, i_0}, we obtain
	$$\max\{2w_{k, i_0}, w_2\} \ge \sum_{a \in A \setminus B_{i_l}} m_a c_{i_0, a} + \sum_{a \in B_{i_0} \cap D_1} m_a c_{i_0, a}.$$
	Let $a \in B_{i_0} \cap D_1.$ Then $a_k \in A_k = \{i_0, i_k\}.$ If $a_k = i_k$, then $a \in B_{i_k} \cap D_1$ and $m_a = 0.$ This gives us 
	$$\sum_{a \in B_{i_0} \cap D_1} m_a c_{i_0, a} = \sum_{\substack{a \in B_{i_0} \cap D_1, \\ a_k = i_0}} m_a c_{i_0, a}.$$
	Since $c_{i_0, a} \ge w_{k, i_0}$ for all $a \in B_{i_0} \cap D_1$ with $a_k = i_0$, it follows that $\max\{2w_{k, i_0}, w_2\} \ge w_2 + w_{k, i_0}$ which contradicts the inequality $0 < w_{k, i_0} < w_2.$ 

    \vspace{2mm}
	\noindent
	\textbf{Case 2.} When $w_1 \le w_{k, i_0} < w_{l, i_0} < w_2.$ Suppose that for some $j \ge s + 1$ with $j \ne p$, it holds that $A_j \subset \{i_k, i_l, i_s\}.$ Then necessarily $A_j = \{i_k, i_s\}$ or $A_j = \{i_l, i_s\}.$ In this scenario, we may proceed exactly as in Case~1 by taking $p = j.$ Therefore, we assume that $A_j \not\subset \{i_k, i_l, i_s\}$ for all $j \ge s + 1$ with $j \ne p.$ For each $1 \le j \le r$, let
	$$f_j = \begin{cases}
		x_{i_0}^{2w_{j, i_0}} & \text{ if } 1 \le j \le k;  \\ 
		\vspace{-2mm} \\
		x_{i_0}^{w_{j, i_0}} x_{i_j}^{w_{j, i_j}} &  \text{ if } k + 1 \le j \le s; \\
		\vspace{-2mm} \\
		x_{i_l}^{2w_{p, i_l}} &  \text{ if } j = p; \\
		\vspace{-2mm} \\
		x_{i_j}^{2w_{j, i_j}} &  \text{ if } s + 1 \le j \le r, ~ j \ne p \text{ and } i_j \in A_j \setminus \{i_0, i_k, i_l, i_s\}.	
	\end{cases}$$
	Then the monomial 
	$$f =  \lcm \left( f_1, \dots, f_r \right) \in I^{(2)}.$$ 
	If $f = x_1^{\al_1} \cdots x_n^{\al_n}$, then $\al_{i_k} = 0$, $\al_{i_0} = \max\{2w_{k, i_0}, w_2\}$ and $\al_{i_s} = w_{s, i_s}.$ We prove that $f \notin I^2.$ To get a contradiction, suppose that $f \in I^2.$ Then, by Lemma \ref{L-MatrixSystem}, there exists $[m_a] \in \mathbb{N}^{1 \times A}$ such that
	\begin{equation}\label{EQN3}
		\sum_{a \in A} m_a = 2
	\end{equation}
	and 
	$$P_I [m_a]^{\top} \le [\al_1 \cdots \al_n]^{\top}.$$
	This implies that
	\begin{equation}\label{IINQ, i_0}
		\sum_{a \in B_{i_0}} m_a c_{i_0, a} \le \max\{2w_{k, i_0}, w_2\},
	\end{equation}
	
	\begin{equation}\label{IINQ, i_k}
		\sum_{a \in B_{i_k}} m_a c_{i_k, a} = 0,
	\end{equation}
	and
	\begin{equation}\label{IINQ, i_s}
		\sum_{a \in B_{i_s}} m_a c_{i_s, a} \le w_{s, i_s}.
	\end{equation}
	Using Inequation \eqref{IINQ, i_k}, we get $m_a = 0$ for all $a \in B_{i_k}.$ Consider the sets $E_1 = \{a \in B_{i_s} : c_{i_s, a} \ge w_{s, i_s}\}$ and $E_2 = \{a \in B_{i_s} : c_{i_s, a} < w_{s, i_s}\}.$ Then we have $B_{i_s} = E_1 \cup E_2.$ Let $a \in A \setminus E_1.$ If $a_s = i_s$, then $a \in B_{i_s}$ and $c_{i_s, a} \ge w_{s, i_s}$, which implies $a \in E_1$, a contradiction. Therefore, $a_s = i_0$, and hence $a \in B_{i_0}$ and $c_{i_0, a} = w_2.$ Consequently, $A \setminus E_1 \subset B_{i_0}.$ By using Inequality \eqref{IINQ, i_s}, we obtain 
	\begin{equation*}\label{INQ8}
		\sum_{a \in E_1} m_a w_{s, i_s} \le \sum_{a \in E_1} m_a c_{i_s, a} + \sum_{a \in E_2} m_a c_{i_s, a} \le w_{s, i_s}.
	\end{equation*}
	This gives $\sum_{a \in E_1} m_a \le 1.$ If $\sum_{a \in E_1} m_a = 0$, then by Equation \eqref{EQN2}, we have $\sum_{a \in A \setminus E_1} m_a = 2.$ In light of Inequality \eqref{IINQ, i_0}, we obtain
	$$\max\{2w_{k, i_0}, w_2\} \ge \sum_{a \in A \setminus E_1} m_a c_{i_0, a} = 2w_2.$$
	This contradicts the inequality $0 < w_{k, i_0} < w_2.$ Hence $\sum_{a \in E_1} m_a = 1.$ By Equation \eqref{EQN3}, we have $\sum_{a \in A \setminus E_1} m_a = 1.$ If $a \in E_1$ with $a_k = i_k$, then $a \in B_{i_k}$, which implies $m_a = 0.$ Thus, we obtain 
	$$\sum_{a \in E_1, a_k = i_0} m_a = \sum_{a \in E_1, a_k = i_0} m_a + \sum_{a \in E_1, a_k = i_k} m_a = \sum_{a \in E_1} m_a = 1.$$
	Since $A \setminus E_1 \subset B_{i_0}$ and $\{a \in A : a_k = i_0\} \subset B_{i_0}$, it follows that
	\begin{align*}
		B_{i_0} = B_{i_0} \cap A = B_{i_0} \cap ((A \setminus E_1) \cup E_1) & = (A \setminus E_1) \cup (B_{i_0} \cap E_1) \\
		& \supset (A \setminus E_1) \cup (\{a \in A : a_k = i_0\} \cap E_1).
	\end{align*}
	Thus, by Inequality \eqref{IINQ, i_0}, we obtain
	$$\max\{2w_{k, i_0}, w_2\} \ge \sum_{a \in A \setminus E_1} m_a c_{i_0, a} + \sum_{a \in E_1, a_k = i_0} m_a c_{i_0, a}.$$
	Note that $c_{i_0, a} \ge w_{k, i_0}$ for all $a \in E_1$ with $a_k = i_0.$ Thus, $\max\{2w_{k,i_0},w_2\} \ge w_2 + w_{k,i_0}.$ This contradicts the inequality $0 < w_{k, i_0} < w_2.$		
\end{proof}

\begin{thm}\label{Th-General}
    With the notation of Set-up \ref{setup}, assume that $A_j \not\subset A_k \cup A_l$ for every $j \in [r] \setminus \{k, l\}$ and for all $(k, l) \in X_I.$ Then $I$ is a Simis ideal if and only if there exists a Simis squarefree monomial ideal $J$ and a standard linear weighting $w$ such that $I = J_w.$ In particular, Conjecture \ref{C2} holds for $I.$
\end{thm}

\begin{proof}
    Suppose there exists a Simis squarefree monomial ideal $J$ and a standard linear weighting $w$ such that $I = J_w.$ Then, in view of \cite[Corollary 5.5]{MVV}, $I$ is Simis. Conversely, if $I$ is Simis, then by Lemma \ref{L-General}, we get $X_I = \emptyset.$ It follows from Remark \ref{R-SLW} that $I = J_w$, where $w$ is a standard linear weighting on $R$ and $J = \sqrt{I}.$ By \cite[Corollary 5.5]{MVV}, $J$ is a Simis ideal.
\end{proof}

The following corollary is a direct consequence of Theorem \ref{Th-General}.

\begin{corollary}
    With the notation of Set-up \ref{setup}, assume that $A_j \not\subset \bigcup_{k \ne j} A_k$ for all $1 \le j \le r.$ Then $I$ is a Simis ideal if and only if there exists a Simis squarefree monomial ideal $J$ and a standard linear weighting $w$ such that $I = J_w.$ In particular, Conjecture \ref{C2} holds for $I.$
\end{corollary}

\begin{thm}\label{Th-2Height}
    With the notation of Set-up \ref{setup}, assume that  $\height(\p) = 2$ for all $\p \in Ass(I).$ Then $I$ is a Simis ideal if and only if there exists a Simis squarefree monomial ideal $J$ and a standard linear weighting $w$ such that $I = J_w.$ In particular, Conjecture \ref{C2} holds for $I.$
\end{thm}	

\begin{proof}
   One can use Lemma \ref{L-2Height} and follow similar steps as in the proof of Theorem~\ref{Th-General}.
\end{proof}

\begin{corollary}\label{C-2Height}
    With the notation of Set-up \ref{setup}, let $\height(\p) = 2$ for all $\p \in \Ass(I)$. Then the following conditions are equivalent:
    \begin{enumerate}[\rm(i)]
        \item $I$ is a Simis ideal.

        \item $I^{(2)} = I^2.$

        \item $G_I$ is a bipartite graph and $I$ has a standard linear weighting.
    \end{enumerate}
\end{corollary}

\begin{proof}
    (i) $\Rightarrow$ (ii) Follows directly.

    (ii) $\Rightarrow$ (iii) Let $J(G_I)$ be the cover ideal of $G_I.$ Then $J(G_I) = \sqrt{I}.$ By Lemma ~\ref{L-2Height}, we get $X_I = \emptyset.$ It follows from Remark \ref{R-SLW} that $I = J(G_I)_w$ for some standard linear weighting $w.$ Using \cite[Corollary 5.5]{MVV}, we obtain that $J(G_I)^{(2)} = J(G_I)^2.$ Now, \cite[Proposition 6.2]{MVV} implies that $G_I$ is a bipartite graph.
    
    (iii) $\Rightarrow$ (i) Let $J \subset R$ be a squarefree monomial ideal, and let $w$ be a standard linear weighting such that $I = J_w.$ Then $J(G_I) = J.$ It follows from \cite[Corollary 2.6]{gitler2005blowup} that $J(G_I)$ is Simis. In view of \cite[Corollary 5.5]{MVV}, we see that $I$ is Simis.
\end{proof}

\section*{Declarations}

\subsection*{Acknowledgments} 

The authors thank the developers of Macaulay2 \cite{M2} for their computational support.

%\subsection*{Funding details} 

%The first author was supported by the Senior Research Fellowship from CSIR-UGC, New Delhi, India. The second author was supported by the Institute Fellowship from the Ministry of Human Resource Development (MHRD), Government of India.

\subsection*{Disclosure statement} 

The authors declare that they have no competing interests.

\subsection*{Data availability statement} 

No datasets were generated or analyzed during the current study. 

\bibliography{ref}

\begin{thebibliography}{10}

\bibitem{BBKM}
Arindam Banerjee, Bidwan Chakraborty, Kanoy~Kumar Das, Mousumi Mandal, and S.~Selvaraja.
\newblock Equality of ordinary and symbolic powers of edge ideals of weighted oriented graphs.
\newblock {\em Comm. Algebra}, 51(4):1575--1580, 2023.

\bibitem{bocci2016waldschmidt}
Cristiano Bocci, Susan Cooper, Elena Guardo, Brian Harbourne, Mike Janssen, Uwe Nagel, Alexandra Seceleanu, Adam~Van Tuyl, and Thanh Vu.
\newblock The waldschmidt constant for squarefree monomial ideals.
\newblock {\em Journal of Algebraic Combinatorics}, 44:875--904, 2016.

\bibitem{bordoloi}
Paromita Bordoloi, Kanoy~Kumar Das, and Rajiv Kumar.
\newblock Support-2 monomial ideals that are simis, 2025.

\bibitem{Chudnovsky1}
G.~V. Chudnovsky.
\newblock Singular points on complex hypersurfaces and multidimensional {S}chwarz lemma.
\newblock In {\em Seminar on {N}umber {T}heory, {P}aris 1979--80}, volume~12 of {\em Progr. Math.}, pages 29--69. Birkh\"auser, Boston, MA, 1981.

\bibitem{Conforti1}
Michele Conforti and G\'erard Cornu\'ejols.
\newblock A decomposition theorem for balanced matrices.
\newblock {\em Integer Programming and Combinatorial Optimization}, 74:147--169, 1990.

\bibitem{Cooper}
Susan~M. Cooper, Robert J.~D. Embree, Huy~T\`ai H\`a, and Andrew~H. Hoefel.
\newblock Symbolic powers of monomial ideals.
\newblock {\em Proc. Edinb. Math. Soc. (2)}, 60(1):39--55, 2017.

\bibitem{DasK}
Kanoy~Kumar Das.
\newblock Equality of ordinary and symbolic powers of some classes of monomial ideals.
\newblock {\em Graphs Combin.}, 40(1):Paper No. 12, 17, 2024.

\bibitem{faridi}
Sara Faridi.
\newblock Monomial ideals via square-free monomial ideals.
\newblock In {\em Commutative algebra}, volume 244 of {\em Lect. Notes Pure Appl. Math.}, pages 85--114. Chapman \& Hall/CRC, Boca Raton, FL, 2006.

\bibitem{Gilter}
Isidoro Gitler, Enrique Reyes, and Rafael~H. Villarreal.
\newblock Blowup algebras of ideals of vertex covers of bipartite graphs.
\newblock In {\em Algebraic structures and their representations}, volume 376 of {\em Contemp. Math.}, pages 273--279. Amer. Math. Soc., Providence, RI, 2005.

\bibitem{gitler2005blowup}
Isidoro Gitler, Enrique Reyes, and Rafael~H Villarreal.
\newblock Blowup algebras of ideals of vertex covers of bipartite graphs.
\newblock {\em Contemporary Mathematics}, 376(273-280):69, 2005.

\bibitem{GRV}
Isidoro Gitler, Enrique Reyes, and Rafael~H. Villarreal.
\newblock Blowup algebras of square-free monomial ideals and some links to combinatorial optimization problems.
\newblock {\em Rocky Mountain J. Math.}, 39(1):71--102, 2009.

\bibitem{GIV}
Isidoro Gitler, Carlos~E. Valencia, and Rafael~H. Villarreal.
\newblock A note on {R}ees algebras and the {MFMC} property.
\newblock {\em Beitr\"age Algebra Geom.}, 48(1):141--150, 2007.

\bibitem{M2}
Daniel~R. Grayson and Michael~E. Stillman.
\newblock Macaulay2, a software system for research in algebraic geometry.
\newblock Available at \url{http://www.math.uiuc.edu/Macaulay2/}.

\bibitem{GGMV}
Gonzalo Grisalde, Jos\'e{} Mart\'inez-Bernal, and Rafael~H. Villarreal.
\newblock Normally torsion-free edge ideals of weighted oriented graphs.
\newblock {\em Comm. Algebra}, 52(4):1672--1685, 2024.

\bibitem{Harbourne}
Brian Harbourne and Craig Huneke.
\newblock Are symbolic powers highly evolved?
\newblock {\em J. Ramanujan Math. Soc.}, 28A:247--266, 2013.

\bibitem{HBook}
J\"urgen Herzog and Takayuki Hibi.
\newblock {\em Monomial ideals}, volume 260 of {\em Graduate Texts in Mathematics}.
\newblock Springer-Verlag London, Ltd., London, 2011.

\bibitem{HJTZ}
J\"urgen Herzog, Takayuki Hibi, Ng\^o{}~Vi\^et Trung, and Xinxian Zheng.
\newblock Standard graded vertex cover algebras, cycles and leaves.
\newblock {\em Trans. Amer. Math. Soc.}, 360(12):6231--6249, 2008.

\bibitem{Hochster}
Melvin Hochster and Craig Huneke.
\newblock Comparison of symbolic and ordinary powers of ideals.
\newblock {\em Invent. Math.}, 147(2):349--369, 2002.

\bibitem{jayanthan2022resurgence}
AV~Jayanthan, Arvind Kumar, and Vivek Mukundan.
\newblock On the resurgence and asymptotic resurgence of homogeneous ideals.
\newblock {\em Mathematische Zeitschrift}, 302(4):2407--2434, 2022.

\bibitem{Toledo2021}
Kazem Khashyarmanesh, Mehrdad Nasernejad, and Jonathan Toledo.
\newblock Symbolic strong persistence property under monomial operations and strong persistence property of cover ideals.
\newblock {\em Bulletin of the Mathematical Society of the Sciences Mathematical of Roumanie (New Series)}, 64(2):105--131, 2021.
\newblock Tome 64 (112), no. 2.

\bibitem{MPD21}
Mousumi Mandal and Dipak~Kumar Pradhan.
\newblock Symbolic powers in weighted oriented graphs.
\newblock {\em Internat. J. Algebra Comput.}, 31(3):533--549, 2021.

\bibitem{MPD}
Mousumi Mandal and Dipak~Kumar Pradhan.
\newblock Comparing symbolic powers of edge ideals of weighted oriented graphs.
\newblock {\em J. Algebraic Combin.}, 56(2):453--474, 2022.

\bibitem{MVV}
Fernando~O. M\'{e}ndez, Maria Vaz~Pinto, and Rafael~H. Villarreal.
\newblock Symbolic powers: Simis and weighted monomial ideals.
\newblock {\em Journal of Algebra and Its Applications}, 24(13n14):2541001, 2025.

\bibitem{Simis}
Aron Simis, Wolmer~V. Vasconcelos, and Rafael~H. Villarreal.
\newblock On the ideal theory of graphs.
\newblock {\em J. Algebra}, 167(2):389--416, 1994.

\bibitem{Skoda}
H.~Skoda.
\newblock Estimations {$L\sp{2}$} pour l'op\'erateur {$\overline \partial $} et applications arithm\'etiques.
\newblock In {\em S\'eminaire {P}ierre {L}elong ({A}nalyse) (ann\'ee 1975/76); {J}ourn\'ees sur les {F}onctions {A}nalytiques ({T}oulouse, 1976)}, volume Vol. 578 of {\em Lecture Notes in Math.}, pages 314--323. Springer, Berlin-New York, 1977.

\bibitem{Waldschmidt}
Michel Waldschmidt.
\newblock Propri\'et\'es arithm\'etiques de fonctions de plusieurs variables. {II}.
\newblock In {\em S\'eminaire {P}ierre {L}elong ({A}nalyse) (ann\'ee 1975/76); {J}ourn\'ees sur les {F}onctions {A}nalytiques ({T}oulouse, 1976)}, volume Vol. 578 of {\em Lecture Notes in Math.}, pages 108--135. Springer, Berlin-New York, 1977.

\end{thebibliography}
\bibliographystyle{plain}

\end{document}